\documentclass[a4paper]{article}

\usepackage{
amsmath,
amsthm,
amscd,
amssymb,
}
\usepackage{comment}
\usepackage{tikz}
\usetikzlibrary{positioning,arrows,calc}
\usepackage{xspace}

\usepackage{graphicx}

\theoremstyle{plain}
\newtheorem{lemma}{Lemma}
\newtheorem{definition}{Definition}
\newtheorem{corollary}{Corollary}

\newtheorem{theorem}{Theorem}

\newtheoremstyle{derp}
{3pt}
{3pt}
{}
{}
{\upshape}
{:}
{.5em}
{}
\theoremstyle{derp}
\newtheorem{example}{Example}

\usepackage{url}

\newcommand{\R}{\mathbb{R}}
\newcommand{\Q}{\mathbb{Q}}
\newcommand{\Z}{\mathbb{Z}}
\newcommand{\C}{\mathbb{C}}
\newcommand{\N}{\mathbb{N}}

\newcommand{\B}{\mathcal{B}}

\newcommand{\ID}{\mathrm{id}}

\newcommand{\lcm}{\mathrm{lcm}}

\newcommand{\bla}{\mbox{\textvisiblespace}}
\newcommand{\ske}{\mathrm{Sk}}
\newcommand{\loc}[1]{#1_{\mathrm{loc}}}

\newcommand{\OC}[1]{\overline{\mathcal{O}(#1)}}

\title{Toeplitz subshift whose automorphism group is not finitely generated}
\author{
Ville Salo \\ vosalo@utu.fi
}

\begin{document}
\maketitle

\begin{abstract}
We compute an explicit presentation of the (topological) automorphism group of a particular Toeplitz subshift with subquadratic complexity. The automorphism group is a non-finitely generated subgroup of rational numbers, or alternatively the $5$-adic integers, under addition, the shift map corresponding to the rational number 1. The group is
\[ (\langle (5/2)^i \;|\; i \in \N \rangle, +) \leq (\Q, +). \]
\end{abstract}

\section{Introduction}

A subshift is a topologically closed, shift-invariant subset of $S^\Z$, where $S$ is a finite alphabet, and shift-invariance means $\sigma(X) = X$ where $\sigma$ is the left shift map. A (topological) endomorphism of a subshift $X$ is a continuous function $f : X \to X$ that commutes with the shift (often required to be surjective). It is an automorphism if it is also a homeomorphism (equivalently, bijective). With respect to function composition, the endomorphisms form a monoid and the automorphisms form a group, and one can ask what kinds of monoids and groups can appear as endomorphism monoids and automorphism groups of subshifts, and how to compute them for of a given subshift. On an infinite subshift, the shift map always generates a copy of $\Z$ in the automorphism group. By the definition of the automorphism group, the shift maps are included in the center of the group, and thus this copy of $\Z$ is a normal subgroup. Often one is also interested in the automorphism group where the group generated by the shift, the shift group, has been quotiented out.

These monoids and groups have been discussed for many (classes of) subshifts. The automorphism group of the full shift was first discussed in \cite{He69}, where it was shown that it contains an isomorphic copy of every finite group, and a copy of the two-generator free group. These results, and stronger ones, were generalized to all mixing subshifts of finite type in \cite{BoLiRu88}. The study of individual endomorphisms and automorphisms of subshifts -- usually full shifts -- is known as the study of cellular automata. From this theory, we obtain many recursion theoretic results about these algebras. Namely, the local rule of a cellular automaton (that is, an endomorphism) gives a canonical recursive presentation of an element of the group, and for example, there is no algorithm to check whether a given element of the automorphism group has finite order \cite{KaOl08}.

Simpler examples of automorphism groups are given by minimal subshifts. For them, the automorphism group can often be computed explicitly. For symbol-to-word substitutions, the shift group is often of finite index in the automorphism group; in particular the automorphism group is virtually cyclic. Such results are shown in for example \cite{He69,Co71,HoPa89,Ol13,SaTo14,CyKr14a}.

For many classes of minimal subshifts, the automorphism group is not only virtually $\Z$, but $\Z$, that is, precisely the shift group. For example, the result of \cite{Ol13} shows that all Sturmian subshifts have an automorphism group consisting of shift maps only. For many examples, a stronger property called minimal self-joinings \cite{Gl03,JuRaSw80,So14} is known, and it is easy to see that subshifts with this property cannot have automorphisms other than shift maps.


Cyr and Kra show many results about automorphism groups of subshifts of subquadratic and linear growth in \cite{CyKr14,CyKr14a}. In \cite{CyKr14a}, they in particular obtain that a linearly recurrent minimal subshift has a virtually $\Z$ automorphism group, answering the question we asked in \cite{SaTo14}. The same result is shown independently with quite different methods by Donoso, Durand, Maass and Petite \cite{DoDuMaPe14}. We do not yet know whether our proof methods, different from both those of \cite{CyKr14} and \cite{DoDuMaPe14}, can be extended to a third proof of this result. Our method in \cite{SaTo14} is hard to apply in abstract settings, since it exploits the explicit self-similarity -- substitutability and desubstitutability -- of the subshift.\footnote{Of course, there is some hope for this: one can give linearly recurrent subshifts (and, to some extent, minimal subshifts in general) a concrete \emph{S-adic} structure as the image of an infinite chain of substitutions with suitable properties, see for example \cite{Du00}.}

However, as another application of the technique of \cite{SaTo14}, we prove the following theorem (see Theorem~\ref{thm:Main}):

\begin{theorem}
\label{thm:Exists}
There exists a Toeplitz subshift with the automorphism group
\[ (\langle (5/2)^i \;|\; i \in \N \rangle, +) \leq (\Q, +). \]
\end{theorem}

This is the additive subgroup of the rationals generated by the powers of $5/2$. What is special about this subshift is that the automorphism group is not finitely generated. While many examples of explicitly computed automorphism groups of minimal subshifts are known, we are not aware of ones that are not virtually cyclic (and thus in particular finitely generated). To prove Theorem~\ref{thm:Exists}, we exploit the fact that a self-similar subshift has a self-similar automorphism group. More precisely, if the self-similarity of a subshift is somehow dictated by a local rule, then we can (sometimes) conjugate automorphisms by this local rule, and reduce their neighborhood (the set of cells the local rule looks at). Characterizing the group then reduces to characterizing the automorphisms of small neighborhood size (in our case, size 1), and how they behave when conjugated by the self-similarity. In \cite{SaTo14}, we had to perform the argument in an extension of the automorphism group; here, the argument can be performed completely within the automorphism group.

Another observation of interest is that our subshift has subquadratic complexity: the number of words of length $n$ that appear in the points of the subshift is $O(n^{1.757})$. It is shown in \cite{CyKr14} that in any transitive system, subquadratic complexity implies that the automorphism group divided by the shift group is periodic, that is, every element has finite order. A natural question, though not discussed in \cite{CyKr14}, is whether one can construct transitive systems where the group divided by the shift group is still infinite. In fact, it is easy to construct such systems, and a better question is whether one can make them minimal. This is explicitly asked in \cite{DoDuMaPe14}, and our example answers this in the positive.

In the case of minimal subshifts, endomorphisms are automatically surjective, and an interesting question is whether they are in fact always injective as well. If this is the case, then the subshift is called coalescent. In \cite{Do97}, an example of a Toeplitz subshift which is non-coalescent is given. Our example is coalescent, which means that our example of an automorphism group is also an example of an endomorphism monoid. In \cite{Su84}, an example of a Toeplitz subshift is constructed which is not uniquely ergodic. It is well-known that a regular Toeplitz subshift is uniquely ergodic, and it is easy to verify that ours is regular. Thus, our example is uniquely ergodic.

Most of this work appeared in the PhD thesis of the author \cite{Sa14}.

\section{Preliminaries}

A \emph{dynamical system} is a pair $(X, T)$ where $X$ is a compact metric space and $T : X \to X$ is a continuous homeomorphism. A \emph{morphism} between dynamical systems $(X, T)$ and $(Y, T')$ is a continuous map $f : X \to Y$ satisfying $f \circ T = T' \circ f$. A morphism $f : X \to Y$ that is surjective is called a \emph{factor map}, and then $Y$ is called a \emph{factor} of $X$. If $f$ is also injective (and thus a homeomorphism) it is called a \emph{conjugacy}, and $X$ and $Y$ are said to be \emph{conjugate}. A dynamical system $(X, T)$ is \emph{minimal} if it contains no proper subsystem $Y \subsetneq X$ with $T(Y) = Y$ other than $Y = \emptyset$. It is \emph{transitive} if for any nonempty open sets $U, V \subset X$ we have $T^i(U) \cap V \neq \emptyset$ for some $i \in \Z$. A minimal dynamical system is transitive.

An \emph{endomorphism} of a dynamical system $(X, T)$ is a morphism $f : X \to X$, and an endomorphism that is homeomorphic is called an \emph{automorphism}. The endomorphisms form a monoid and the automorphisms form a group, with respect to function composition. The functions $\{T^i \;|\; i \in \Z\}$ form a subgroup of the automorphism group, called the \emph{shift group}.

The symbol $S$ denotes a finite alphabet with the discrete topology. By $S^\Z$ we mean the set of two-way infinite sequences, called \emph{points}. We write $S^*$ for the set of ($0$-indexed) finite words over $S$, including the empty word. Our indexing conventions are $x_i$ for the symbol $x(i)$ in the $i$th coordinate of $x$, and $x_{[j,k]}$ for the subword $x_j x_{j+1} \cdots x_k$, and in this notation, the topology of $S^\Z$ is given by the metric $d(x, y) = \inf \{2^{-i} \;|\; x_{[-i,i]} = y_{[-i,i]}\}$. We use similar conventions for indexing words. By $|v|_u$ we denote the number of (possibly overlapping) occurrences of $u$ in $v$, and $u \sqsubset v$ means $v_{[i,i+|u|-1]} = u$ for some $i$. When $u \in S^*$, we write $u^\Z = x$ where $x_i = u_{i \bmod |u|}$ for all $i \in \Z$.

The \emph{shift map} $\sigma : S^\Z \to S^\Z$ is defined by $x_i = x_{i+1}$. A \emph{subshift} is a topologically closed subset $X$ of $S^\Z$ which is \emph{shift-invariant} in the sense that $\sigma(X) = X$. Then $(X, \sigma)$ is a dynamical system. For a point $x \in S^\Z$, we write $\mathcal{O}(x) = \{\sigma^i(x) \;|\; i \in \Z\}$ for the \emph{orbit} of $x$. For any point $x \in S^\Z$, the \emph{orbit closure} $\OC{x}$ is a subshift. A point $x \in S^\Z$ is \emph{uniformly recurrent} if for each $n$ there exists $m$ such that $x_{[-n,n]} \sqsubset x_{[i-m,i+m]}$ for all $i \in \Z$. The orbit closure of a uniformly recurrent point is minimal.

A \emph{Toeplitz point} is a point $x \in S^\Z$ such that for all coordinates $i$ there exists a \emph{period} $p > 0$ such that $x_{i+kp} = x_i$ for all $k \in \Z$. Then, for all intervals $[j,j']$, we also find $p > 0$ such that $x_{[j,j'] + kp} = x_{[j,j']}$ for all $k \in \Z$, and we similarly say $p$ is the period of $[j, j']$ in $x$. A Toeplitz point is clearly uniformly recurrent, so the subshift $X \subset S^\Z$ it generates is minimal. A \emph{Toeplitz subshift} is any subshift generated by a Toeplitz point. Not every point in a Toeplitz subshift need be Toeplitz, but the Toeplitz points are necessarily dense, since the orbit of the point generating the subshift is dense, and contains only Toeplitz points.

Note that periodic points $u^\Z$ are Toeplitz with our definition, and they generate Toeplitz subshifts conjugate to finite systems $(\Z_n, (x \mapsto x+1))$. In fact, these finite systems are important in what follows. Thus, we will consider $\Z_n$ a dynamical system with dynamics $x \mapsto x+1$. These are precisely the finite minimal dynamical systems. We will also discuss them with trivial dynamics, $(\Z_n, \ID_{\Z_n})$ (in which case they are of course not subshifts if $n > 1$).

Every morphism $f : X \to Y$ between subshifts $X, Y \subset S^\Z$ has a \emph{radius} $R \in \N$ and a \emph{local rule} $\loc{f} : S^{[-R,R]} \to S$ such that $f(x)_i = \loc{f}(x_{[i-R,i+R]})$ for all $x \in X$, $i \in \Z$. A morphism is said to have \emph{one-sided radius} $R$ and \emph{one-sided local rule} $\loc{f} : S^{[0,R]} \to S$ if $f(x)_i = \loc{f}(x_{[i,i+R]})$ for all $x$ and $i$. If $f$ has radius $R$, then $\sigma^R \circ f$ has one-sided radius $2R$.

\section{Toeplitz subshifts and independence}

In this section, we also define the concept of independence, and characterize it for Toeplitz subshifts in Theorem~\ref{thm:AllAboutPerp}. In the following sections, we mainly need that disjointness of Toeplitz systems implies mutually independence.

We begin with a general discussion of Toeplitz subshifts. Our methods are very elementary, and we do not need much of the theory of Toeplitz subshifts -- in particular, while the discussion below is strongly related to the maximal equicontinuous factor of a Toeplitz subshift, we will not discuss this factor explicitly, but only its finite factors.\footnote{The maximal equicontinuous factor of a Toeplitz subshift is just an inverse limit of the finite factors, so the difference is small.} The lemmas that follow can be extracted from, or found directly in, any reference that discusses the maximal equicontinuous factor of a Toeplitz subshift \cite{Ku03,Wi84,Do05}, but we give a self-contained presentation with a more combinatorial point of view.

\begin{definition}
Let $x \in S^\Z$ be arbitrary. For each $k > 0$, define the \emph{$k$-skeleton of $x$} as the point
\[ \ske(k, x)_i = \left\{\begin{array}{ll}
x_i,  & \mbox{if } \forall m \in \Z: x_{i + mk} = x_i, \mbox{ and} \\
\bla, & \mbox{otherwise.}
\end{array}\right. \]
The number $k$ is an \emph{essential period} of $x$ if $\sigma^\ell(\ske(k, x)) \neq \ske(k, x)$ for all $0 < \ell < k$.
\end{definition}

We note that since $\ske(k, x)$ necessarily has period $k$, the condition $\sigma^\ell(\ske(k, x)) \neq \ske(k, x)$ for all $0 < \ell < k$ means exactly $|\OC{\ske(k, x)}| = k$.

The $k$-skeleton of $x$ contains the coordinates having period $k$, and the least period of a coordinate $i$ in $x$ is the least $k$ such that $\ske(k, x)_i \neq \bla$. It is easy to see that $x$ is Toeplitz if and only if $\lim_{k \rightarrow \infty} \ske(k!, x) = x$.

We show that while the definition talks about global periods, one can detect the periods locally in Toeplitz subshifts. This is well-known, see for example Proposition~4.71 in \cite{Ku03}.

\begin{lemma}
\label{lem:ContinuityThing}
Let $X \in S^\Z$ be a Toeplitz subshift. Then for all $k > 0$, there exists $m \geq 0$ such that for all $y \in X$, we have $y_j = y_{j+k} = \cdots = y_{j+(m-1)k}$ if and only if $\ske(k, y)_j = y_j$. In particular, the function $f : X \to (S \cup \{\bla\})^\Z$ defined by $f(y) = \ske(k, y)$ is a block map.
\end{lemma}

\begin{proof}
Let $x \in X$ be Toeplitz. We claim that there exists $m$ such that in any point $y \in \OC{x} = X$, for all $h \in [0, k-1]$ either
\[ \forall \ell \in \Z: y_{h+\ell k} = y_h, \]
that is, the arithmetic progression $h + \Z k$ is constant, or
\[ \forall \ell: \exists j \in [1, m-1]: y_{h + \ell k} \neq y_{h+(\ell+j)k}, \]
that is, the arithmetic progression $h + \Z k$ has a syndetic set of coordinates where the symbol changes.

Suppose that this is not the case. Then for all $m$ we find $y\in \OC{x}$ such that
\[ y_{h + \ell k} = y_{h + (\ell + 1) k} = \cdots = y_{h + (\ell + m - 1) k} \neq y_{h + (\ell + m) k} \]
for some $\ell$ and $h$. Since $y \in \OC{x}$, we in particular find $j_m \in \Z$ such that
\[ a = x_{j_m} = x_{j_m+k} = \cdots = x_{j_m+(m-1)k} \neq x_{j_m + mk}. \]

We show that this cannot happen for arbitrarily large $m$. Namely, consider such $j_m$ modulo $k$. For $h \in [0, k-1]$, if $x_{h+\ell k} = a$ for all $\ell$, then we cannot have $j_m = h \bmod k$ for any $m$. Otherwise, $x_{h+\ell k} = b \neq c = x_h$ for some $\ell$. Let $p$ be a period for $\{h, h+\ell k\}$ in $x$, so that $x_{h+np} = c$ and $x_{h+\ell k+np} = b$ for all $n$. Then if $m > \lcm(p, k) + \ell$, we cannot have $j_m = h \bmod k$. This rules out all values of $j_m$ modulo $h$, which is a contradiction.

This shows that $f$ is a block map: $f(y)_j = y_j$ if $y_j = y_{j+k} = \cdots = y_{j+mk}$, and $f(y)_j = \bla$ otherwise.
\end{proof}

Thus, in a Toeplitz point $x$, and also the points in its orbit closure, we can locally detect skeletons, and the local detection rule is simply to check that the cell has a particular period for some fixed amount of steps. Since $k$ is an essential period of $x$ if and only if $|\OC{\ske(k, x)}| = k$, we have some obvious corollaries.

\begin{lemma}
Let $x, y \in S^\Z$ be Toeplitz points that generate the same subshift $X$. Then the essential periods of $x$ are the same as those of $y$.
\end{lemma}

\begin{proof}
Since the points generate the same subshift, in particular $y \in \OC{x}$. For $k$ and $x$, let $f$ be the block map $f$ from the previous lemma. We have
\[ |\OC{\ske(k, y)}| = |\OC{f(y)}| = |f(\OC{y})| = |f(\OC{x})| = |\OC{f(x)}| = |\OC{\ske(k, x)}|, \]
so $k$ is an essential period of $y$ if and only if it is one of $x$.
\end{proof}

The following is a trivial case of Theorem~1.3 from \cite{Do05}.

\begin{lemma}
\label{lem:FiniteFactors}
Let $X$ be a Toeplitz subshift. Then $\Z_n$ is a factor of $X$ is and only if $n | k$ for some essential period $k$ of~$X$.
\end{lemma}

\begin{proof}
Let $x \in S^\Z$ be a Toeplitz point generating $X$. If $k$ is an essential period, then $\OC{\ske(k, x)}| = k$. Clearly, $\OC{\ske(k, x)}$ is then conjugate to $\Z_k$, so also $\Z_n$ is a factor of $X$ for all $n | k$.

Next, suppose $f : X \to \Z_n$ is a factor map. We may assume it is one-sided and $f(x) = 0$. By continuity, $y_{[0, r]}$ determines the image $f(y)$ for $y \in X$. Let $w$ be such that $f(y) = 0$ whenever $y_{[0,|w|-1]} = w$. Then $w$ occurs at $x_{[i, i+|w|-1]}$ only if $i \equiv 0 \bmod n$. Let $k$ be minimal such that $w$ occurs in $\ske(k, x)$. Then $k$ is an essential period because $\sigma^{\ell}(\ske(k, x)) = \ske(k, x)$ implies that there is an $\ell$-periodic subsequence of $\Z$ where $x$ contains only occurrences of $w$, so that $w$ occurs in $\ske(\ell, x)$. We have $n | k$ because $w$ occurring in $\ske(k, x)$ implies that there is a $k$-periodic subsequence of $\Z$ where $x$ contains only occurrences of $w$, and $f(y) = 0$ if $y_{[0,|w|-1]} = w$.
\end{proof}

By the previous lemma, the finite factors of a Toeplitz subshift give all the information about the periodic structure of the generating point $x$. Usually, this information is organized into the maximal equicontinuous factor,\footnote{In the spectral theory point of view, it is collected into the set of \emph{eigenvalues}: For example, we can say $\lambda \in \C$ is an \emph{eigenvalue} of $X$, if there is a continuous function $\phi : X \to \C \setminus \{0\}$, called an \emph{eigenfunction} for $\lambda$, with $\phi(\sigma(x)) = \lambda \phi(x)$. Then, in the Toeplitz case, one can show that any eigenvalue $\lambda$ must be an $n$th root of unity where $\Z_n$ is a finite factor, and conversely if $\Z_n$ is a finite factor, we can find an eigenfunction for $\lambda = e^{2\pi/n}$.} based on the following lemma.

\begin{lemma}
\label{lem:PeriodicStructure}
Let $x \in S^\Z$ be Toeplitz and aperiodic. Then there exists a sequence $n_1, n_2, \ldots$ of essential periods of $x$ such that $n_1 > 1$, for all $i$ we have $n_i < n_{i+1}$ and $n_i | n_{i+1}$, and $\lim_i \ske(n_i, x) = x$.
\end{lemma}

The sequence $(n_1, n_2, \ldots)$ is called the \emph{periodic structure}, and the \emph{maximal equicontinuous factor} is the inverse limit of the diagram $\cdots \rightarrow \Z_{n_2} \rightarrow \Z_{n_1}$ where the map $\Z_{n_{i+1}} \rightarrow \Z_{n_i}$ is $a \mapsto a \bmod \Z_{n_i}$. The reason the maximal equicontinuous factor is useful in organizing this information is that the inverse limit is independent of the choice $(n_1, n_2, \ldots)$ up to isomorphism. It turns out that this is indeed the maximal equicontinuous factor, in the sense that it is a factor with equicontinuous dynamics, and every factor with equicontinuous dynamics factors through it. For us, it is enough to talk directly about the set of finite factors of the form $\Z_n$.

\begin{lemma}
\label{lem:LeastPeriodThing}
If $X$ is a Toeplitz subshift, $x \in X$ is Toeplitz, and the least period of $[i, j]$ in $x$ is $k$, then $\Z_k$ is a factor of $X$.
\end{lemma}

\begin{proof}
Clearly, $\ske(x, k)_{[i, j]} = x_{[i, j]}$. If $\sigma^\ell(\ske(k, x)) = \ske(k, x)$ for some $0 < \ell < k$, then $\ell$ is a smaller period of $[i, j]$ in $x$. Thus, $k$ is an essential period, and the result follows from Lemma~\ref{lem:FiniteFactors}.
\end{proof}

We now define the notions of disjointness and independence. Disjointness is a relatively well-known concept in the theory of dynamical systems, and it was introduced in \cite{Fu67}. We do not know if independence has been studied previously, but it is useful to us when studying the automorphism group of our Toeplitz example in Section~\ref{sec:InterestingEndos}. While distinct for minimal subshifts in general, we show in Theorem~\ref{thm:AllAboutPerp} that the two notions, disjointness and independence, are equivalent for Toeplitz subshifts, and simply state that the systems have no common finite factor. These, and some other equivalent notions, are listed in Theorem~\ref{thm:AllAboutPerp}.

\begin{definition}
If $X,Y$ are two subshifts, we say that a subshift $J \subset X \times Y$ is a joining of $X$ and $Y$ if the restrictions of the projection maps $\pi_1 : J \to X$ and $\pi_2 : J \to Y$ are surjective. If each joining is equal to $X \times Y$, we then say that $X$ and $Y$ are \emph{disjoint}, and denote this by $X \perp Y$.
\end{definition}

\begin{lemma}
\label{lem:PerpIsMinimal}
Suppose $X$ and $Y$ are minimal. Then $X \perp Y$ if and only if $X \times Y$ is minimal.
\end{lemma}

\begin{proof}
If $X \times Y$ is minimal and $J$ is a joining of $X$ and $Y$, then $J$ is a nonempty subshift of $X \times Y$, and thus $J = X \times Y$ by minimality, so that $X \perp Y$.

Suppose then that $X \perp Y$ and $J$ is a nonempty subshift of $X \times Y$. Then $(x, y) \in J$ for some $x \in X$ and $y \in Y$. Since both $X$ and $Y$ are minimal, $x$ generates $X$ and $y$ generates $Y$, so that the orbit closure $K$ of $(x, y)$ projects onto $X$ through $\pi_1$ and onto $Y$ through $\pi_2$. By $X \perp Y$, we have $K = X \times Y$. Of course, we then have $J = X \times Y$, so that $X \times Y$ is minimal.
\end{proof}

In the case that one of the systems is finite, we have the following alternative characterization.

\begin{lemma}
\label{lem:FinitePerp}
Let $X$ be minimal. Then $X \perp \Z_m$ if and only if $(X, \sigma^m)$ is minimal.
\end{lemma}

\begin{proof}
We have $X \perp \Z_m$ if and only if $X \times \Z_m$ is minimal.

If $(X, \sigma^m)$ is minimal, then for all $\epsilon > 0$ and $x, y \in X$, if $(x, n_1), (y, n_2) \in X \times \Z_n$ and $n_1 \leq n_2$ (the other case being symmetric), by the minimality of $\sigma^m$, there exists $k$ such that $d(\sigma^{km}(\sigma^{n_2-n_1}(x)), y) < \epsilon$, and then
\[ d(\sigma^{km+n_2-n_1}((x, n_1)), (y, n_2)) = d((\sigma^{km+n_2-n_1}(x), n_2), (y, n_2)) < \epsilon. \]

If $X \times \Z_m$ is minimal then for any $\epsilon > 0$ and $x, y \in X$, for some $n \in \N$ we have $\sigma^n(x, 0) = (z, 0)$ where $d(z, y) < \epsilon$. Clearly $n = km$ for some $k$, so $(X, \sigma^m)$ is minimal.
\end{proof}

\begin{definition}
A morphism $\phi : X \times Y \to Z$ is \emph{right-independent} if $\phi$ factors through the projection map $\pi_1 : X \times Y \to X$, that is,
\[ \exists h : X \to X: \phi = h \circ \pi_1. \]
We define \emph{right-dependence} as the complement of right-independence, and \emph{left-independence} and \emph{left-dependence} symmetrically. If all morphisms $\phi : X \times Y \to X$ are right-independent then we say $X$ is independent from $Y$. If $X$ is independent from $Y$ and $Y$ from $X$, then we say the two are \emph{mutually independent}.
\end{definition}

In general, we can define these notions in categories with products, and in concrete categories where products correspond to set theoretic products, $X$ is independent from $Y$ if there are no morphisms $\phi : X \times Y \to X$ which actually depend on the $Y$-coordinate.

The two notions have nontrivial interplay within the class of minimal systems. We can at least construct two minimal systems $X$ and $Y$ such that $X \times Y$ is minimal, but $X$ depends on $Y$:


\begin{example}
For all $u \in \{0, 1\}^*$, let $O(u)$ be the word where odd coordinates of $u$ have been flipped (counting from the left, starting with $0$), and $E(u)$ the word where even coordinates have been flipped. Let $B(u) = E(O(u))$. If $|u|$ is odd, then $O(uv) = O(u)E(v)$ and $E(uv) = E(u)O(v)$.

Let $w_0 = 000$, and inductively define $w_{i+1} = w_i w_i O(w_i) O(w_i) B(w_i)$, all of which are of odd length. For example, $w_1 = 000000010010111$ and
\begin{align*}
w_2 = \;\,&000000010010111 000000010010111 010101000111101 \, \cdot \\
&010101000111101 1111111011011000.
\end{align*}

For all $i$, $w_i$ occurs in all of $w_{i+1}$, $O(w_{i+1})$, $E(w_{i+1})$ and $B(w_{i+1})$:
\begin{itemize}
\item $O(w_{i+1}) = O(w_i)E(w_i) w_iB(w_i)E(w_i)$,
\item $E(w_{i+1}) = E(w_i)O(w_i) B(w_i)w_iO(w_i)$, and
\item $B(w_{i+1}) = B(w_i)B(w_i) E(w_i)E(w_i)w_i$.
\end{itemize}
For any $i$, the point $x = \lim_j w_j$ is an infinite product of the words $w_{i+1}$, $O(w_{i+1})$, $E(w_{i+1})$ and $B(w_{i+1})$. By the previous observation, it is then uniformly recurrent. Thus, the system $X = \overline{\mathcal{O}(x)}$ with the shift dynamics $\sigma$ is minimal. Since $w_iw_i \sqsubset X$ for all $i$ and $|w_i|$ is odd, also $\sigma^2$ is minimal.

Now, let $Y = \mathcal{O}((01)^\Z)$. It follows from the minimality of $\sigma^2$, Lemma~\ref{lem:FinitePerp} and Lemma~\ref{lem:PerpIsMinimal} that also $X \times Y$ is minimal. By the inductive definition of $X$, the map
\[ \phi(x, y) = x + y, \]
where $+$ is cellwise addition modulo $2$, is well-defined from $X \times Y$ to $X$. It clearly depends on the $Y$-coordinate. \qed
\end{example}

We give an obvious composition result. Of the two claims we prove, we only need the first one.\footnote{There are many more symmetric versions of this lemma. We can of course replace right by left, but we can also define a dual notion of `coindependence' by considering maps from $X$ to the coproduct (disjoint union) $X \cup Y$ instead of maps from the product $X \times Y$ to $X$.}

\begin{definition}
A map $\xi : X \times Y \to Z$ is \emph{right-surjective} if for all $x$, the function $\xi|_{\{x\} \times Y} : \{x\} \times Y \to Z$ is surjective. We define \emph{left-surjectivity}, \emph{right-injectivity} and \emph{left-injectivity} in the obvious way, and \emph{bi-surjectivity} and \emph{bi-injectivity} as the conjunction of the respective left- and right notions.
\end{definition}

\begin{lemma}%
\label{lem:FactorThing}
Let $\xi : X \times Y \to Z$ and $\xi' : X \times Z \to X$ be morphisms. If
\begin{itemize}
\item $\xi$ is right-surjective and $\xi'$ right-dependent, then $X$ is dependent of $Y$.
\item $\xi$ is right-dependent and $\xi'$ right-injective, then $X$ is dependent of $Y$.
\end{itemize}
\end{lemma}

\begin{proof}
Define $\phi(x, y) = \xi'(x, \xi(x, y))$. If either assumption holds for $\xi$ and $\xi'$, this map shows that $X$ is dependent of $Y$.
\end{proof}

\begin{lemma}%
\label{lem:DependsOnBoring}
A nontrivial subshift $X$ is dependent of every system $(\Z_m, \ID)$ with $m > 1$.
\end{lemma}

\begin{proof}
Let $X$ be any such subshift. There exist two distinct endomorphisms $\phi_1$ and $\phi_2$ of $X$, for example, $\ID_X$ and the shift map. Let $\emptyset \subsetneq C \subsetneq \Z_m$ be any subset. Let 
\[ \phi(x, y) = \left\{ \begin{array}{ll}\phi_1(x), &\mbox{if $y \in C$,} \\ \phi_2(x), &\mbox{otherwise.} \\\end{array}\right. \]
Then $\phi$ is a right-dependent map, so $X$ is not independent of $(\Z_m, \ID)$.
\end{proof}

\begin{definition}
A \emph{(nontrivial) invariant} of a system $X$ is a factor map from $X$ to a system $(\Z_m, \ID)$. 
\end{definition}

\begin{lemma}%
\label{lem:InvariantNotTransitive}
If $X$ is transitive, then it has no nontrivial invariant.
\end{lemma}

\begin{proof}
The image of a transitive system in a factor map is transitive.
\end{proof}

\begin{lemma}%
\label{lem:DependsOnInvariant}
If $X$ is nontrivial and $X \times Y$ has a nontrivial right-surjective invariant, then $X$ is dependent of $Y$.
\end{lemma}

\begin{proof}
Let $\xi : X \times Y \to (\Z_m, \ID)$ be a nontrivial right-surjective invariant. By Lemma~\ref{lem:DependsOnBoring}, $X$ depends on $(\Z_m, \ID)$, so that some map $\xi' : X \times (\Z_m, \ID) \to X$ is right-dependent. The result then follows from Lemma~\ref{lem:FactorThing}.
\end{proof}

\begin{theorem}
\label{thm:AllAboutPerp}
Let $X, Y$ be nontrivial Toeplitz subshifts. Then the following are equivalent:
\begin{enumerate}
\item	$X \perp Y$ \label{perp}
\item	$X \times Y$ is minimal \label{minimal}
\item	$X \times Y$ is transitive \label{transitive}
\item	$X$ and $Y$ are mutually independent \label{mutindep}
\item	$X$ is independent from $Y$ \label{indep}
\item	$X$ and $Y$ have no common nontrivial finite factors. \label{nofactors}
\end{enumerate}
\end{theorem}

\begin{proof}
The equivalence of $\eqref{perp}$ and $\eqref{minimal}$ was proved in Lemma~\ref{lem:PerpIsMinimal}. 
It is clear that $\eqref{transitive}$ follows from $\eqref{minimal}$ and $\eqref{indep}$ follows from $\eqref{mutindep}$. If $\eqref{nofactors}$ does not hold, then $X$ and $Y$ have a common finite factor $\Z_m$ through factor maps $\phi_1 : X \to \Z_m$ and $\phi_2 : Y \to \Z_m$ (since the systems systems $\Z_m$ are the only minimal finite systems). This means $\xi(x, y) \mapsto \phi_1(x) - \phi_2(y)$ is a bi-surjective invariant, so that $\eqref{transitive}$ does not hold by Lemma~\ref{lem:InvariantNotTransitive}, and $\eqref{indep}$ does not hold by Lemma~\ref{lem:DependsOnInvariant}.

We now tackle the hard part, the implications $\eqref{nofactors} \implies \eqref{minimal}$ and $\eqref{nofactors} \implies \eqref{mutindep}$, which conclude the proof.

So, suppose $\eqref{nofactors}$. We first show that $\eqref{mutindep}$ follows. Let $\xi : X \times Y \to X$ be a morphism with (one-sided) radius $R$. Choose $w \in \B_{R+1}(X)$ and $u, u' \in \B_{R+1}(Y)$ arbitrarily. Fix Toeplitz points $x \in X$ and $y \in Y$ such that $x_{[0,R]} = w$ and $y_{[0, R]} = u$ (using the fact that Toeplitz points are dense) and choose $j \in \N$ such that $y_{[j, j+R]} = u'$. Of course, $z = \xi(x, y) \in X$ is Toeplitz, since $x$ and $y$ are. Let $[0, R]$ have least period $k_x$ in $x$, let $0$ have least period $k_z$ in $z$, and let $[0, j+R]$ have least period $k_y$ in $y$.

By Lemma~\ref{lem:PeriodicStructure}, there exists a factor map from $X$ to both $\Z_{k_x}$ and $\Z_{k_z}$, and from $Y$ to $\Z_{k_y}$. This means that $\gcd(k_x k_z, k_y) = 1$ by the assumption that $X$ and $Y$ have no common finite factors.


If $\gcd(k_x k_z, k_y) = 1$, then there exists $m$ such that $m k_x k_z \equiv j \bmod k_y$, so
\[ \loc{\xi}(w, u) = \xi(x,y)_0 = \sigma^{m k_x k_z}(\xi(x, y))_0 = \xi(x, \sigma^{m k_x k_z}(y))_0 = \loc{\xi}(w, u'). \]
Because $w$, $u$ and $u'$ were chosen arbitrarily, $\xi$ is right-independent. Thus $\eqref{nofactors} \implies \eqref{mutindep}$.

Next, we prove $\eqref{minimal}$ assuming $\eqref{nofactors}$, along similar lines: Fix Toeplitz points $x \in X$ and $y \in Y$. Let $R \in \N$ be arbitrary and let $w \in \B_{R+1}(X)$ and $u \in \B_{R+1}(Y)$ be arbitrary words. Let $j_1, j_2$ be such that $x_{[j_1, j_1+R]} = w$ and $y_{[j_2, j_2+R]} = u$. As previously, the least period $k_x$ of $[0, j_1+R]$ in $x$ is coprime with the least period $k_y$ of $[0, j_2+R]$ in $y$. Thus, there exists $m$ such that $m k_x \equiv j_2-j_1 \bmod k_y$. We have
\[ \sigma^{n k_x k_y + m k_x + j_1}(x, y)_{[0, R]} = (w, u) \]
for all $n \in \N$. Thus, the orbit of $(x, y)$ is dense in $X \times Y$. Since $(x, y)$ is Toeplitz, $X \times Y$ is minimal. 
\end{proof}

An important observation about finite factors is that endomorphisms of a Toeplitz subshift $X$ induce maps on the finite factors.

\begin{lemma}
\label{lem:InducedMapInFactor}
Let $\chi : X \to \Z_n$ be a morphism, where $X$ is minimal, and let $f : X \to X$ be an arbitrary morphism. Then there exists a morphism $f_n : \Z_n \to \Z_n$ such that $\chi \circ f = f_n \circ \chi$.
\end{lemma}

\begin{proof}
Let $\chi(x) = 0$ and $\chi(f(x)) = k$. Then $\chi(\sigma^j(x)) = j$ and $\chi(f(\sigma^j(x))) = k+j$ for all $j$. Since $\OC x = X$, we have $\chi(f(y)) = \chi(y) + k$ for all $y \in X$, so we can take $f_n(a) = a+k$ for all $a \in \Z_n$.
\end{proof}

In fact, the maps $f_n$ associated to $f$ determine it completely, see Proposition~1 in \cite{Ol13}.

\section{Toeplitz substitutions}

One way to generate Toeplitz sequences is the following type of substitution process. Let $w \in (S \dot\cup \{\bla\})^*$. We say $w$ is a \emph{partial word} over the alphabet $S$, and $\bla$ represents a \emph{missing coordinate}. We build a point $x(w)$ by a recursive process that starts with the point $\bla^\Z$ and fills the gaps of the current point with $w^\Z$, repeatedly. Let $\phi$ be the map that, given $y$ and $z$ in $(S \dot\cup \{\bla\})^\Z$, where both tails of $y$ contain infinitely many missing coordinates, writes $z$ in the missing coordinates of $y$. More precisely, if $\ell \geq 0$ is the least nonnegative coordinate of $y$ with $y_\ell = \bla$, we define
\[ \phi(y, z)_j = \left\{\begin{array}{ll}
y_j &\mbox{if } y_j \neq \bla, \\
z_0 &\mbox{if } j = \ell, \\
z_k &\mbox{if } {y_j = \bla} \, \wedge j > \ell \wedge k = |y_{[\ell,j-1]}|_{\bla}, \mbox{ and} \\
z_k &\mbox{if } {y_j = \bla} \, \wedge j < \ell \wedge k = |y_{[j+1,\ell]}|_{\bla}. \\
\end{array}\right. \]
If $w_0, w_{|w|-1} \in S$, then writing $\psi_w'(x) = \phi(x, w^\Z)$, we define
$x^i(w) = (\psi_w')^i(\bla^\Z)$ for all $i$, and define $x(w) = \lim_i x^i(w)$. It is easy to see that the limit $x(w)$ exists, and that indeed $x(w) \in S^\Z$, that is, this point contains no missing coordinates. It is also clearly a Toeplitz point. Our example is of the form $X_w = \overline{\mathcal{O}(x(w))} \subset S^\Z$. In this notation, the main result of this article is that the group $(\langle (5/2)^i \;|\; i \in \N \rangle, +) \leq (\Q, +)$ mentioned in the abstract is the automorphism group of the subshift $X_{1 \bla 0 \bla 0}$.

\newcommand{\zr}{\hspace{2.75pt}0 \hspace{2.75pt}}
\newcommand{\yx}{\hspace{2.75pt}1 \hspace{2.75pt}}

We note that defining, instead, $\psi_w(x) = \phi(w^\Z, x)$, we obtain the same limit $\lim_{n \rightarrow \infty} \psi_w^n(\bla^\Z)$. In fact, starting with the point $\bla^\Z$ and applying the operations $\psi_w'$ and $\psi_w$ in any order, we obtain the same limit point $x(w)$.

We note some obvious properties of the substitution map.

\begin{lemma}
\label{lem:PhiContinuity}
Let $w \in (S \cup \{\bla\})^*$, $|w| = p$ and $|w|_{\bla} = q$. Let $X = X_w$ and $\psi_w(x) = \phi(w^\Z, x) : S^\Z \to S^\Z$. Then $\psi_w$ is continuous and injective, and for any $k \in \Z$, we have the equality
\[ \sigma^{kp} \circ \psi_w = \psi_w \circ \sigma^{kq}. \]
\end{lemma}

We say $p > 0$ is a \emph{lazy period} of a partial point $y \in (S \cup \{\bla\})^\Z$ if $y_i = y_{i+kp}$ whenever $y_i, y_{i+kp} \in S$ for $i,k \in \Z$. The interpretation of having lazy period $p$ is that there is a way to fill the $\bla$-gaps so that the resulting point has period $p$. We note that having lazy periods $j$ and $j'$ does not imply the lazy period $\gcd(j, j')$ in general. For example, $(01\bla\bla\bla\bla)^\Z$ has lazy periods $2$ and $3$, but not $1$. However, the following is true.

\begin{lemma}
\label{lem:LazyPeriods}
If $y \in (S \cup \{\bla\})^\Z$ has period $j$ and lazy period $j'$, the $y$ has lazy period $\gcd(j, j')$.
\end{lemma}

\begin{proof}
Suppose $y_i \neq y_{i+kp}$, where $y_i, y_{i+kp} \in S$ for $i,k \in \Z$ where $p = \gcd(j, j')$. Let $p = mj - m'j'$ where $m,m' \in \N$. Since $y$ has period $j$, we have $y_{i + kmj} = y_i$. Since we have lazy period $j'$, if $a = y_{i + kmj - km'j'} = y_{i + kp}$, we must have $a = \bla$ or $a = y_{i + kmj} = y_i$. A contradiction, since $a = y_{i+kp} \in S \setminus \{y_i\}$.
\end{proof}

\begin{lemma}
\label{lem:Important}
Let $w \in (S \cup \{\bla\})^p$ with $w_0,w_{|w|-1} \neq \bla$, and suppose $w^\Z$ has least lazy period $p$ for prime $p > 1$, and $|w|_{\bla} = q$. Then the essential periods of $x(w)$ are $p^i$ for $i \in \N$.
\end{lemma}

\begin{proof}
Let $x = x(w)$ and $x^i = x^i(w)$ for all $i$.

First, we verify by induction that $x^j$ has period $p^j$ for all $j$ (for which no assumptions on $p$ and $q$ are needed): This is true for $x^1$ because $x^1 = w^\Z$ and $|w| = p$. By definition, $x^{j+1} = \phi(x^j, w^\Z)$. Let $i \in \N$ be arbitrary. If $x^j_i \neq \bla$, then also $x^{j+1}_i = x^{j+1}_{i + p^{j+1}}$ because $p^j$ is a period of $x^j$. Since $x^j$ has period $p^j$, we have $|x^j_{[i, i + p^{j+1} - 1]}|_{\bla} = kp$ for $k = |x^j_{[i, i + p^j - 1]}|_{\bla}$. Since $w^\Z$ has period $p$, we then have $x^{j+1}_i = x^{j+1}_{i+p^{j+1}}$.

Next, we show by induction on $j$ that every coordinate $i$ such that $x^{j+1}_i \neq \bla$ and $x^j_i = \bla$ in fact has least lazy period $p^{j+1}$, and $|x^{j+1}_{[0,p^{j+1}-1]}|_{\bla} = q^{j+1}$. For $j = 0$, consider a coordinate $i$ such that $x^1_i \neq \bla$. It has period $p$, and thus lazy period $p$. If its least lazy period is $m$, then $m | p$ by the previous lemma, and thus $m = p$. By the assumption on $w$, we have $|x^1_{[0,p-1]}|_{\bla} = |w|_{\bla} = q$.

Inductively on $j$, consider a coordinate $i$ such that $x^{j+1}_i \neq \bla$ and $x^j_i = \bla$. Such a coordinate has period $p^{j+1}$ in $x^{j+1}$. Let $m$ be its least lazy period. Then $m | p^{j+1}$ as before, so that $m = p^\ell$ for some $\ell \leq j+1$. We claim that $m = p^{j+1}$. Suppose the contrary. Then, the coordinate $i$ in particular has lazy period $p^j$ in $x^{j+1}$, so that $x^{j+1}_{i+np^j} = \bla$ or $x^{j+1}_{i+np^j} = x^{j+1}_i$ for all $n$. By induction, and the periodicity of $x^j$, we have $|x^j_{[i+(n-1)p^j, i+(np^j-1]}|_{\bla} = q^j$ for all $n$. Then there exists $h \in [0,p-1]$ such that $x^{j+1}_i = w_h$ and $x^{j+1}_{i+np^j} = w_{h+nq^j}$ for all $n$, where $w$ is indexed modulo $p$. Since $\gcd(q^j, p^j) = 1$, $x^{j+1}_{i+np^j}$ takes on all coordinates of $w$. Thus, $w_i \in \{a, \bla\}^p$ for some $a \in S$, so that $w^\Z$ has lazy period $1$, a contradiction.

We have essentially shown that $\ske(p^j, x) = x^j$: Coordinates $x^j_i \neq \bla$ have period $p^j$ in $x^j$ and thus also $x$, so by definition, $\ske(p^j, x)_i = x^j_i$ for such $i$. On the other hand, if $x^j_i = \bla$, then $x^k_i \neq \bla$ for some $k > j$, and then $i$ has least lazy period $p^k > p^j$ in $x^k$, so that certainly its period in $x$ cannot be less than $p^k$, so that again $\ske(p^j, x)_i = \bla = x^j_i$.

We claim that $p^j$ is an essential period for each $j$. Otherwise, $\sigma^\ell(x^j)  = x^j$ for some $0 < \ell < p^j$. Take any coordinate $i$ with $x^j_i \neq \bla$ but $x^{j-1}_i \neq \bla$. Then $\sigma^\ell(x^j)  = x^j$ implies $i$ has period, in particular lazy period, $\ell$ in $x^j$. But the least lazy period of such $i$ in $x^j$ is $p^j$.
\end{proof}

When the assumptions of Lemma~\ref{lem:Important} hold, Lemma~\ref{lem:PhiContinuity} can also be strengthened.

\begin{lemma}
\label{lem:Homeomorphism}
Let $w \in (S \cup \{\bla\})^p$ with $w_0, w_{|w|-1} \neq \bla$, and suppose $w^\Z$ has least lazy period $p$ for prime $p > 1$, and $|w|_{\bla} = q$. Let $\chi : X \to \Z_p$ be the unique morphism with $x(w) \mapsto 0$. Then $\psi_w(x) = \phi(w^\Z, x) : S^\Z \to S^\Z$ restricts to a homeomorphism from $X$ to $\chi^{-1}(0)$.
\end{lemma}

\begin{proof}
Since $\psi_w$ is continuous and injective and $S^\Z$ is compact, we only need to show $\psi_w(X) = \chi^{-1}(0)$. First, suppose $y \in \chi^{-1}(0)$, so that $y \in X$ and $\chi(y) = 0$. Clearly, there is a unique sequence $y' \in S^\Z$ such that $y = \psi_w(y')$. We claim that $y' \in X$. Since $\chi$ is a local rule identifying the $n$-skeleton $x^1(w) = w^\Z$, and $y$ is a limit point of $x(w)$, we must have $\sigma^{j_i}(x(w)) \rightarrow y$ for some $j_1, j_2, \ldots$. Since $\chi(y) = 0$, we may restrict to a subsequence so that $\chi(\sigma^{j_i}(x(w))) = 0$. Thus, by the inductive definition of $x(w)$, the $\bla$-coordinates of the $p$-skeleton of $y$ are filled with words of $x$, so that $y' \in X$.

Then, let us show $\psi_w(y) \in \chi^{-1}(0)$ for all $y \in X$. Let $y \in X$ be arbitrary. By Lemma~\ref{lem:FinitePerp}, $\sigma^q$ is minimal on $X$, so that there is a sequence $j_1, j_2, \ldots$ such that $\sigma^{j_i q}(x(w)) \rightarrow y$. Then by the definition of the substitution process, we have $\psi_w(y) = \lim_{i \rightarrow \infty} \sigma^{j_i p} x(w)$.
\end{proof}

We of course have infinitely many holes after any step of the substitution process. Under a simplifying assumption, these holes become separated in a uniform way.\footnote{This shows that the fibers of the projection to the maximal equicontinuous factor are either singletons or of cardinality $2$.}

\begin{lemma}
\label{lem:Sparse}
Let $w \in (S \cup \{\bla\})^p$ with $w_0 \neq \bla$ and $\bla\bla \not\sqsubset w$. Then if $i' > i$ and $x^j(w)_{[i, i']} \in \bla S^{i'-i-1} \bla$, we have $i'-i \geq 2^j$.
\end{lemma}

\begin{proof}
Note that $w_0 \neq \bla$ and $\bla\bla \not\sqsubset w$ together imply $\bla\bla \not\sqsubset w^\Z$.

Since $\bla\bla \not\sqsubset w^\Z$, the minimal distance (the quantity $i' - i$) between two distinct symbols $\bla$ in $x^1(w)$ is at least $2$. If the minimal distance between two symbols $\bla$ in $x^j(w)$ is at least $2^j$, then the minimal distance between two symbols $x^{j+1}(w)$ is at least $2^{j+1}$ since $\bla\bla \not\sqsubset w^\Z$.
\end{proof}

\section{Groups}

We discuss the groups we will implement as endomorphism monoids.

\begin{definition}
\label{def:A}
For $m, n \in \N$, we define a subgroup of $(\Q, +)$ by
\[ A(n, m) = \left\langle \left(\frac{n}{m}\right)^i \;|\; i \in \N \right\rangle. \]
\end{definition}

\begin{lemma}
The group $A(n, m)$ is not finitely generated if $m \nmid n$.
\end{lemma}

\begin{proof}
If $m \nmid n$, then $m$ has a prime divisor $p$ with $p^k | m$ but $p^k \nmid n$ for some $k$. Then $|\frac{n}{m}|_p = \ell \geq p$, where $|\cdot|_p$ denotes the $p$-adic norm. Then $|\left(\frac{n}{m}\right)^i|_p = \ell^i$, so that $A(n, m)$ is not bounded in the $p$-adic norm. The result follows, since the $p$-adic numbers under addition form an ultrametric group in the sense that $|a + b|_p \leq \max(|a|_p, |b|_p)$
\end{proof}

\begin{definition}
Let $G$ be an abelian group with generators $\{g_i \;|\; i \in \N\}$ such that $n g_i = m g_{i+1}$ for all $i$. Then $G$ is said to be \emph{$(n, m)$-lifting}.
\end{definition}

\begin{lemma}
The group $A(n, m)$ is 
$(n, m)$-lifting.
\end{lemma}

\begin{proof}
We choose generators $g_i = \left(\frac{n}{m}\right)^i$, and then
\[ n g_i = m g_{i+1} \iff \left(\frac{n}{m}\right) g_i = g_{i+1} \iff \left(\frac{n}{m}\right)\left(\frac{n}{m}\right)^i  = \left(\frac{n}{m}\right)^{i+1}. \]
\end{proof}


\begin{lemma}
\label{lem:GroupStuff}
Let $(n, n', m) \in \N^3$ satisfy
\begin{equation}
\label{eq:Nice}
n = 2n' + 1, 1 < m \leq n' \mbox{ and } \gcd(m, n) = 1,
\end{equation}
and let $G = \langle g_i \;|\; i \in \N \rangle$ be $(n, m)$-lifting. Then every element $g \in G$ can be written as
\[ g = k_1 g_1 + k_2 g_2 + \ldots + k_j g_j, \]
where $k_i \in [-n', n']$ for all $i$, and $k_j \neq 0$. In the group $A(n, m)$, there is a unique such representation for each $g \in G$. Conversely, if there is a unique such representation for all $g \in G$, then $G$ is isomorphic to $A(n, m)$.
\end{lemma}

\begin{proof}
All elements of $G$ can be put into such form by first adding a suitable multiple of $n g_1 - m g_2 = 0$ to reduce $k_1$, then $n g_2 - m g_3 = 0$ to reduce $k_2$, and so on. This process eventually terminates because $m \leq n'$.

If the form is not unique for some element of the group, then by subtracting two distinct but equivalent forms and putting the result in the normal form, we obtain
\[ k_j g_j + k_{j+1} g_{j+1} + \cdots + k_{j'} g_{j'} = 0 \]
where $k_j \neq 0$, $k_{j'} \neq 0$ and $k_i \in [-n', n']$ for all $i$, and $j \leq j'$, that is, there is at least one nonzero coefficient. Letting $G = A(n, m)$ with generators assigned as in the previous lemma, the equation above cannot hold:
\begin{align*}
(k_j g_j + \cdots + k_{j'} g_{j'})m^{j'} &=
\left(k_j\left(\frac{n}{m}\right)^{j} + k_{j+1}\left(\frac{n}{m}\right)^{j+1} + \cdots + k_{j'}\left(\frac{n}{m}\right)^{j'}\right)m^{j'} \\
&= k_j n^j m^{j'-j} + k_{j+1} n^{j+1} m^{j'-j-1} + \cdots + k_{j'} n^{j'} \\
&\equiv k_j n^j m^{j'-j} \bmod{n^{j+1}} \\
&\not\equiv 0 \bmod{n^{j+1}}
\end{align*}
since $\gcd(m, n) = 1$ and $k_j \in [-n',n']$.

Let $G = A(n, m)$, and let $H = \langle h_i \;|\; i \in \N \rangle$ be another $(n, m)$-lifting group where the representations in normal form are unique. We define a map $\phi$ from $G$ to $H$ by mapping $g_i \mapsto h_i$, and in general mapping
\[ k_1 g_1 + k_2 g_2 + \ldots + k_j g_j \mapsto k_1 h_1 + k_2 h_2 + \ldots + k_j h_j, \]
when the left side is in the normal form. It is clear that this is a bijection between the groups, since we assumed that the representations exist and are unique on both sides. To see that it is a homomorphism, note that if $g, h \in G$, then the unique normal form for $g + h$ is obtained by summing the components of the normal forms of $g$ and $h$, and applying the algorithm described in the first paragraph of the proof. By applying the same transformations to the representation of $\phi(g) + \phi(h)$ obtained by summing the representations of $\phi(g)$ and $\phi(h)$, we obtain precisely $\phi(g + h)$, and thus $\phi(g + h) = \phi(g) + \phi(h)$.
\end{proof}

\section{The example and its automorphism group}
\label{sec:InterestingEndos}

We can now construct our example of a Toeplitz subshift whose automorphism group is not finitely generated. For concreteness, we construct the group $A(5, 2)$. More generally, for any triple $(p,p',q)$ satisfying \eqref{eq:Nice} and $p$ prime, we will find a Toeplitz subshift whose automorphism group (in fact, the whole endomorphism monoid) is isomorphic to the group $A(p, q)$.\footnote{The group $A(5, 2)$ is the simplest interesting example obtained like this. The construction also applies to $A(3, 1)$, but this group is isomorphic to $\Z$.}

We make some standing assumptions for the rest of this section. We fix a triple $(p, p', q)$ satisfying \eqref{eq:Nice} with $p$ prime. We also fix a word $w \in (S \cup \{\bla\})^p$ and the subshift $X = X_w$, with the properties
\begin{itemize}
\item $w_0, w_{|w|-1} \neq \bla$ and $\bla\bla \not\sqsubset w$,
\item $w^\Z$ has least lazy period $p$,
\item $|w|_{\bla} = q$, and
\item $\ID_X$ is the only radius $0$ block map, or \emph{symbol map}, on $X$.
\end{itemize}
Since $p$ is prime, $w^\Z$ has least lazy period $p$ as long as $w$ contains two distinct letters. We also have $\bla\bla \not\sqsubset w^\Z$. One possible word is $w = 1(\bla 0)^q0^{2(p'-q)}$, and setting $p=5$, $q=2$ we obtain $w = 1 \bla 0 \bla 0$. Let $\chi : X \to \Z_p$ be the unique factor map with $x(w) \mapsto 0$. In the following, $w$, $X$, $p$, $p'$, $q$ and $\chi$ are thought of as fixed.

\begin{lemma}
\label{lem:PeriodicStructureOfx}
With the standing assumptions, the finite factors of $X$ are the systems $\Z_{p^\ell}$ where $\ell \in \N$.
\end{lemma}

\begin{proof}
The word $w$ satisfies the assumptions of Lemma~\ref{lem:Important}, so the essential periods of $x(w)$ are the numbers $p^\ell$. The result follows from Lemma~\ref{lem:FiniteFactors}.
\end{proof}

\begin{lemma}
\label{lem:mnPerp}
With the standing assumptions, we have that $X \perp \Z_{q^j}$ for all $j$. In particular, $(X, \sigma^{q^j})$ is minimal for all $j$.
\end{lemma}

\begin{proof}
The finite factors of $\Z_{q^j}$ are the systems $\Z_\ell$ where $\ell | q^j$. The finite factors of $X$, on the other hand, are the systems $\Z_{p^i}$ for $i \in \N$ by Lemma~\ref{lem:FiniteFactors}. Since $\gcd(q,p) = 1$, there are no common finite factors, and then $X \perp \Z_{q^j}$ by Theorem~\ref{thm:AllAboutPerp}. The latter claim follows from Lemma~\ref{lem:FinitePerp}.
\end{proof}

\begin{lemma}[Pasting Lemma \cite{Mu00}]
Suppose $A_1, A_2, \ldots, A_k$ are closed subsets of a topological space $X$, $f_i : A_i \to Y$ are continuous functions on the $A_i$, and for all $i, j \in [1, k]$, we have $f_i(x) = f_j(x)$ for all $x \in A_i \cap A_j$. Then the function $f : \bigcup_i A_i \to Y$ defined by $f(x) = f_i(x)$ for all $i \in [1,k]$ and $x \in A_i$ is continuous.
\end{lemma}

To an endomorphism $f$ of $X \subset S^\Z$, we associate the function ${\downarrow f}$ which applies $f$ in the unknown coordinates, that is, the coordinates that do not come from the $p$-skeleton $w^\Z$.

\begin{definition}
For a block map $f : X \to X$, define its corresponding \emph{unlifted} map ${\downarrow f} : \chi^{-1}(i) \to S^\Z$ by
\[ {\downarrow f} = \sigma^i \circ \psi_w \circ f \circ \psi_w^{-1} \circ \sigma^{-i} \]
and then ${\downarrow f} : X \to S^\Z$ by joining the (disjoint) domains.
\end{definition}

\begin{lemma}
\label{lem:Unlifting}
With the standing assumptions, for any endomorphism $f : X \to X$, the unlifted map $\downarrow f$ maps $X$ to $X$, and the codomain restriction $\downarrow f : X \to X$ is an endomorphism of $X$. Furthermore, $\downarrow f(\chi^{-1}(i)) = \chi^{-1}(i)$.
\end{lemma}

\begin{proof}
The function $\downarrow f$ is a continuous map from $X$ to $X$ by the pasting lemma, since the $\chi^{-1}(i)$ are closed, and the partial definitions of $\downarrow f$ are continuous maps on $\chi^{-1}(i)$ by following the chain of domains and codomains (and in particular applying Lemma~\ref{lem:Homeomorphism}). To show that this is an endomorphism of $X$, we only need to show that it commutes with the shift. If $y \in \chi^{-1}(i)$ for $i < p-1$, then
\begin{align*}
{\downarrow f}(\sigma(y)) &= \sigma^{i+1} \circ \psi_w \circ f \circ \psi_w^{-1} \circ \sigma^{-i-1} (\sigma(y)) \\
&= \sigma \circ \sigma^i \circ \psi_w \circ f \circ \psi_w^{-1} \circ \sigma^{-i} (y) \\
&= \sigma({\downarrow f}(y))
\end{align*}
since $\sigma(y) \in \chi^{-1}(i+1)$.

If $y \in \chi^{-1}(p-1)$, then $\chi(\sigma(y)) = 0$. Thus by Lemma~\ref{lem:PhiContinuity},
\begin{align*}
{\downarrow f}(\sigma(y)) &= \psi_w \circ f \circ \psi_w^{-1} (\sigma(y)) \\
&= \psi_w \circ f \circ \psi_w^{-1} \circ \sigma^{p} \circ \sigma^{-p+1}(y) \\
&= \psi_w \circ f \circ \sigma^{q} \circ \psi_w^{-1} \circ \sigma^{-p+1}(y) \\
&= \psi_w \circ \sigma^{q} \circ f \circ \psi_w^{-1} \circ \sigma^{-p+1}(y) \\
&= \sigma^p \circ \psi_w \circ f \circ \psi_w^{-1} \circ \sigma^{-p+1}(y) \\
&= \sigma (\sigma^{p-1} \circ \psi_w \circ f \circ \psi_w^{-1} \circ \sigma^{-p+1}(y)) \\
&= \sigma({\downarrow f}(y)).
\end{align*}

For the last claim, $y \in \chi^{-1}(i)$ implies
\[ {\downarrow f}(y) = \sigma^i \circ \psi_w \circ f \circ \psi_w^{-1} \circ \sigma^{-i}(y). \]
Since $\psi_w(z) \in \chi^{-1}(0)$ for all $z \in X$, ${\downarrow f}(y) \in \chi^{-1}(i)$.
\end{proof}

The following is now easy to verify.

\begin{lemma}
With the standing assumptions, for any endomorphisms $g, h : X \to X$ we have
\[ {\downarrow (g \circ h)} = {\downarrow g} \circ {\downarrow h}. \]
\end{lemma}

\begin{lemma}
\label{lem:FixedPoint}
Let $0 < a < 1$ and $b \in \R$ be arbitrary. If $R \in \R$ is large enough, then $aR + b < R$.
\end{lemma}

Next, we show that not only can endomorphisms be unlifted, but also lifted once shifted. The important property of this lifting is that it decreases the radius if the radius is large.

\begin{lemma}
\label{lem:InductionStep}
With the standing assumptions, there exists $r$ such that for every endomorphism $f : X \to X$ with radius $R \geq r$, there exists a morphism $h : X \to X$ with radius less than $R$ such that
$f = {\downarrow h} \circ \sigma^{-k}$
for some $0 \leq k < p$.
\end{lemma}

\begin{proof}
Suppose $R$ is the radius of $f$, and $R \geq p$. Let $k$ with $0 \leq k < p$ satisfy $\chi \circ (f \circ \sigma^k) = \chi$ (see Lemma~\ref{lem:InducedMapInFactor}), and let $f_1 = f \circ \sigma^k$. Then $f_1$ has radius less than $R + p$. Let $g_1 : X \times \Z_q \to X$ be defined by
\[ g_1(y, i) = (\sigma^i \circ \psi_w^{-1} \circ f_1 \circ \psi_w \circ \sigma^{-i})(y). \]

Since all the maps in the composition are continuous, $g_1$ is continuous. To check that $g_1$ is a morphism, we now only have to check that it is shift-commuting. The calculation is very similar to the one in Lemma~\ref{lem:Unlifting}. If $i < q-1$, then this is true basically by definition:
\begin{align*}
g_1(\sigma(y), i+1) &= (\sigma^{i+1} \circ \psi_w^{-1} \circ f_1 \circ \psi_w \circ \sigma^{-i-1})(\sigma(y)) \\
&= \sigma((\sigma^i \circ \psi_w^{-1} \circ f_1 \circ \psi_w \circ \sigma^{-i})(y)) \\
&= \sigma(g_1(y, i)).
\end{align*}

If $i = q-1$, then
\begin{align*}
g_1(\sigma(y), 0) &= (\psi_w^{-1} \circ f_1 \circ \psi_w)(\sigma(y)) \\
&= (\sigma^q \circ \psi_w^{-1} \circ f_1 \circ \psi_w \circ \sigma^{-q})(\sigma(y)) \\
&= \sigma((\sigma^{q-1} \circ \psi_w^{-1} \circ f_1 \circ \psi_w \circ \sigma^{-q+1})(y)) \\
&= \sigma(g_1(y, q-1)).
\end{align*}
The second equality follows from Lemma~\ref{lem:PhiContinuity}.

Because $X \perp \Z_q$, it follows from Theorem~\ref{thm:AllAboutPerp} that $X$ is independent of $\Z_q$. Thus, there exists a map $h_1 : X \to X$ such that $g_1(y, i) = h_1(y)$ for all $y \in X$ and $i \in \Z_q$. But then, by the definition of the unlifting operation, we have $f_1 = {\downarrow h_1}$. It is easy to see that $h_1$ can be taken to have the same radius as $g_1$, since $h_1(x) = g_1(x, 0)$ for all $x \in X$.

Now, let us compute an upper bound for the radius of $h_1$, that is, we need to determine $\ell$ such that $x_{[-\ell,\ell]}$ determines $h_1(x)_0$ for $x \in X$. Let $j$ be minimal such that $w_j = \bla$. Then since
\[ h_1(x) = g_1(x, 0) = \psi_w^{-1} \circ f_1 \circ \psi_w, \]
by the definition of $\phi_w$ and $\chi \circ f_1 = \chi$ we have $h_1(x)_0 = f_1(\psi_w(x))_j$
If $x_{[-\ell,\ell]}$ is known, then the word $\psi_w(x)_{[-\ell',\ell']}$ is uniquely determined for at least any $\ell' \leq \frac{\ell p}{q}-p$, and $f_1(\phi(y, x))_j$ is determined if $\ell' \geq R + p + j$ , where $R + p$ is the upper bound for the radius of $f_1$. Thus, we can take
\[ \ell = \left\lceil \frac{q(R + 3p)}{p} \right\rceil \]

In particular, for some constant $b$, $h_1$ has radius at most $\frac{q}{p}R + b$. By Lemma~\ref{lem:FixedPoint}, this is smaller than $R$ if $R$ is large enough.
\end{proof}

We write $\sigma_j = {\downarrow^j \sigma}$, so that $\sigma_0 = \sigma$, and for $j \geq 1$, $\sigma_j(y)$ equals $y$ in the coordinates $i$ where $\ske(p^j, y)_i = y_i$, and the subsequence of coordinates $i$ with $\ske(p^j, y)_i = \bla$ is shifted one step to the left.

\begin{lemma}
\label{lem:Abelian}
With the standing assumptions, for all $j, k \in \N$,
\[ \sigma_j \circ \sigma_k = \sigma_k \circ \sigma_j. \]
\end{lemma}

\begin{proof}
Suppose $j < k$. Then 
\begin{align*}
\sigma_j \circ \sigma_k &= {\downarrow^j} (\sigma \circ \downarrow^{k - j} \sigma) \\
&= {\downarrow^j} (\downarrow^{k - j} \sigma \circ \sigma) \\
&= \sigma_k \circ \sigma_j.
\end{align*}
\end{proof}

\begin{lemma}
\label{lem:Recursion}
With the standing assumptions, for any endomorphism $h : X \to X$, if $n > 0$ and
\[ h = \sigma_0^{\ell_0} \circ \sigma_1^{\ell_1} \circ \cdots \circ \sigma_{n-1}^{\ell_{n-1}} \circ {\downarrow^n h}, \]
then $h = \sigma^i$ for some $i \in \Z$.
\end{lemma}

\begin{proof}
First, note that we can make $n$ as large as we like by repeatedly substituting this expression for $h$ on the right-hand side and using ${\downarrow (g \circ h)} = {\downarrow g} \circ {\downarrow h}$. Let $g = \sigma_0^{\ell_0} \circ \sigma_1^{\ell_1} \circ \cdots \circ \sigma_{n-1}^{\ell_{n-1}}$.


If $R$ is the radius of $h$, we take $n$ larger than $\lceil \log (2R + 1) \rceil$, so that by Lemma~\ref{lem:Sparse}, if $x^n(w)_{[i, i']} \in \bla S^{i'-i-1} \bla$ we have $i'-i \geq 2^{2R+1}$. Choose a point $y$ with $\ske(p^n, y)_0 = \bla$ and define a function $\xi : X \to X$ where
\[ \xi(x) = \phi(\ske(p^n, y), x) \]
for all $x \in X$, that is, $\xi(x)$ is the point where the $\bla$-coordinates of $\ske(p^n, y)$ are replaced by symbols of $x$ in order, so that in particular $\xi(y)_0 = x_0$.

Now, we note that there exists $k$ such that $g(\xi(z))_k = z_0$ for all $z \in X$. Namely, whatever $z \in X$ is, it will be shifted in the exact same way, as a subsequence of $\xi(z)$, by all the maps $\sigma_i^{\ell_i}$ for $i < n$, since by definition, $\sigma^i$ simply shifts the contents of the $\bla$-coordinates of $\ske(p^n, y)$ to the left, jumping over other coordinates.

The equation $({\downarrow^n h})(\xi(z)) = \xi(h(z))$ holds basically by the definition of the $\downarrow$-operation, as ${\downarrow^n h}$ applies $h$ to the subsequence of $y$ found in the $\bla$-coordinates of $\ske(p^n, y)$. Thus, since we assumed $h = g \circ {\downarrow^n h}$, the equality
\[ h(\xi(z))_k = (g \circ {\downarrow^n h}) (\xi(z))_k = g(\xi(h(z)))_k = h(z)_0 \]
holds for all $z \in X$.

Now, we have $|\ske(p^n, y)_{[k-R,k+R]}|_{\bla} \leq 1$, by the assumption on $n$. If there is no $t \in {[k-R,k+R]}$ with $\ske(p^n, y)_t = \bla$, then $h$ is a constant map:
\[ \forall z \in X: h(z)_0 = h(\xi(z))_k = \loc{h}(y_{[k-R,k+R]}). \]
Due to the assumption that $X$ supports no symbol maps other than the identity, this is impossible. If there is such $t$, then we note that, by the definition of $\xi$, there exists $i$ such that for all $z \in X$, $\xi(z)_t = z_i$. Then,
\[ \forall z \in X: h(z)_0 = h(\xi(z))_k = \loc{h}(y_{[k-r,t-1]},z_i,y_{[t+1,k+r]}), \]
so $h = \sigma^i \circ \pi$ for some symbol map $\pi$. Again, by the assumption that $X_w$ has no symbol maps other than the identity, we have that $h$ is a shift map.
\end{proof}

\begin{theorem}
\label{thm:Main}
For every tuple $(p,p',q)$ satisfying \eqref{eq:Nice} with $p$ prime, there exists a two-way Toeplitz subshift whose endomorphism monoid is isomorphic to the group $A(p, q)$.
\end{theorem}

\begin{proof}
With the standing assumptions and the notation above (in particular, choosing a suitable $w$ corresponding to $p$ and $q$ as discussed in the beginning of this section), we show that the endomorphism monoid of $X$ is
$\langle \sigma_i \;|\; i \in \N \rangle$,
and it is isomorphic to the group $A(p, q)$ by the isomorphism $\sigma_i \mapsto \left( \frac{p}{q} \right)^i$. By Lemma~\ref{lem:PhiContinuity} and Lemma~\ref{lem:Abelian}, the $\sigma_i$ indeed generate an $(p,q)$-lifting group.

Given any $f : X \to X$, we start iterating Lemma~\ref{lem:InductionStep} on $f$ to obtain
\[ f = \sigma^{k_1} \circ {\downarrow h_1} = \sigma^{k_1} \circ {\downarrow (\sigma^{k_2} \circ {\downarrow h_2})} = \sigma^{k_1} \circ {\downarrow (\sigma^{k_2} \circ {\downarrow (\sigma^{k_3} \circ {\downarrow h_3})})} = \cdots, \]
which can be rewritten, using the equality ${\downarrow (g \circ h)} = {\downarrow g} \circ {\downarrow h}$, as
\[ f = \sigma_0^{k_1} \circ {\downarrow h_1} = \sigma_0^{k_1} \circ \sigma_1^{k_2} \circ {\downarrow^2 h_2} = \sigma_0^{k_1} \circ \sigma_1^{k_2} \circ \sigma_2^{k_3} \circ {\downarrow^3 h_3} = \cdots. \]
By Lemma~\ref{lem:InductionStep}, the radii of the $h_i$ eventually decrease below some constant $r$, and then for some $n$, we have $h_n = h_{n+m}$ for some $m > 0$.

Then we have 
\[ h_n = \sigma_0^{k_{n+1}} \circ \sigma_1^{k_{n+2}} \cdots \sigma_{m-1}^{k_{n+m}} \circ {\downarrow^m h_n}. \]
It follows from Lemma~\ref{lem:Recursion} that $h_n$ is a shift map, and then $f \in \langle \sigma_i \;|\; i \in \N \rangle$.

Now, we have seen that there are no endomorphisms other than the maps in $\langle \sigma_i \;|\; i \in \N \rangle$. To see that this group is isomorphic to $A(p, q)$, by Lemma~\ref{lem:GroupStuff} it is enough to show
\[ \sigma_j^{k_j} \circ \sigma_{j+1}^{k_{j+1}} \circ \cdots \circ \sigma_{j'}^{k_{j'}} = \ID_X \]
where $k_j \neq 0$, $k_{j'} \neq 0$ and $k_i \in [-p', p']$ for all $i$ is impossible. But this is again clear from the periodic structure of $X$: $\sigma_j^{k_j}$ shifts the $j$th level of its input, and does not change other coordinates.
\end{proof}

\begin{corollary}
There exists a minimal Toeplitz subshift whose automorphism group is not finitely generated.
\end{corollary}

\subsection{Some comparisons and additional information}

We give some comparisons with existing literature and answer some natural questions about the subshift.

First, we discuss the approach of \cite{Co71} for attacking the automorphism group, which is the main tool also in the more recent papers about automorphism groups \cite{Ol13,DoDuMaPe14}: this is the study of the automorphism group through the induced group on the maximal equicontinuous factor.

Toeplitz subshifts are \emph{almost 1-1 extensions} of their maximal equicontinuous factors, that is, the factor map always has at least one singleton fiber, and thus a dense set of such fibers \cite{Su84}. In such a case, every automorphism is uniquely determined by the map it induces on the maximal equicontinuous factor, see for example Proposition 1 of \cite{Ol13}. The maximal equicontinuous factor of our example $X$ is the $5$-adic adding machine $Y$, and the endomorphisms of such a system are easily seen to be additions of $5$-adic integers. Thus, to compute the endomorphism monoid of our example, one only needs to compute the $5$-adic numbers $c$ such that the endomorphism $y \mapsto y+c$ of the maximal equicontinuous factor lifts to a continuous map on $X$.

The group of such $c$ is indeed precisely the subgroup of the rationals shown in the abstract, considered as a subgroup of the group of $5$-adic integers. By the algebraic properties of the $n$th order shift maps we define, these numbers $c$ must be precisely those obtained as finite $\Z$-linear combinations of the $5$-adic numbers $1$, $\frac{5}{2} = \ldots 22230$, $\left(\frac{5}{2}\right)^2 = \ldots 1113400$ and so on. The special structure of the subshift guarantees that such maps on $Y$ lift to continuous maps on $X$, but no other maps do. Unlike, for example, \cite{Co71,Ol13}, we do not work out the automorphism group by studying the fibers of this factor map, but work directly with the $n$th order shift maps. This seems like a natural approach for this particular subshift: while the $5$-adic expansions of these numbers look somewhat complicated, the $n$th order shift maps satisfy natural relations (that is, they form a lifting group).

One might ask whether there is a simple way to see what the automorphism group of our example is through the study of the set of fibers. We do not know whether this is the case, but mention some simple observations. First, it is a simple consequence of Lemma~\ref{lem:Sparse} that in our example, the fibers of the factor map to the maximal equicontinuous factor are all of cardinality $1$ or $2$. The set of fibers with cardinality $2$ changes in a somewhat complicated way whenever the holes in $w$ are moved, but the automorphism group is not affected by such movements, in the sense we have much freedom in the choice of $w$ in the proof of Theorem~\ref{thm:Main}, for large $p$ and $q$, but the automorphism group is a function of $p$ and $q$ only.

We now discuss the connection between our example and the result of \cite{CyKr14} stating that a transitive subshift whose language has a subquadratic growth has a periodic automorphism group, where a group $G$ is \emph{periodic} if
\[ \forall g \in G: \exists n: g^n = 1.\]
Namely, we show that our Toeplitz subshift has subquadratic growth.


\begin{lemma}
Suppose $T : \R \to \R$ is nondecreasing and $T(n) \leq aT(n/b + c)$ for all large enough $n$, where $a \geq 1, b > 1, \log_b a > 1$ and $c \in \N$. Then $T(n) = O(n^{\log_b a})$.
\end{lemma}

\begin{proof}
We have
\[ T(n) \leq aT(n/b + c) \leq a^2 T(n/b^2+c/b+c) \leq a^3 T(n/b^3+c/b^2+c/b+c) \leq \cdots, \]
so that $T(n) \leq a^k T\left(n/b^k + c\frac{1-b^{-k}}{1-b^{-1}}\right)$ for all $k$. Setting $k = \log_b n$ we have
\[ T(n) \leq a^{\log_b n} T\left(1 + c\frac{1-b^{-k}}{1-b^{-1}}\right) = O(n^{\log_b a}). \]
\end{proof}

\begin{lemma}
For $w = 1\bla0\bla0$, $X_w$ has subquadratic growth.
\end{lemma}

\begin{proof}
Let us count the number $n(k)$ of words of length $k$ that occur in $X_w$ for large $k$. Let $\chi$ have right radius $r$, so that a word of length $r+1$ has a unique phase, in the sense that it matches a unique subpattern of $(1 \bla 0 \bla 0)^\Z$ in a unique way. There are $5$ possible phases, so if $k > r$ there are at most $5 \ell$ words of length $n$, where $\ell$ is the number of words of length $\lceil\frac{2k}{5}\rceil + 10$ (a trivial upper bound for the number of holes left after filling the skeleton), that is,
\[ n(k) \leq 5 n\left(\left\lceil\frac{2k}{5}\right\rceil + 10\right). \]

By the previous lemma, setting $a = 5$, $b = \frac{5}{2}$ and $c = 10$, we have
\[ n(k) = O(n^{\log_b a}) = O(k^{1.757}). \]
\end{proof}

Since $X_w$ has subquadratic growth, the result of \cite{CyKr14} should hold, and indeed it does. Namely, the automorphism group we obtained is periodic when the shift maps are quotiented out: the automorphism group of $X_w$ is isomorphic to $A(p,q)$ where the subgroup $\Z$ corresponds to the shift maps. Since the group is abelian, and for the generators $\left(\frac{5}{2}\right)^i$ we have $2^i\left(\frac{5}{2}\right)^i \in \Z$ for all $i$, the group is of the required form. In particular, this shows that while the subshifts of the type considered in \cite{CyKr14} always have a periodic automorphism group up to powers of the shift, they need not have a finite automorphism group up to the shift.

\section*{Acknowledgements}

The author would like to thank Luca Zamboni for suggesting the study of automorphism groups of Toeplitz subshifts, Ilkka T\"orm\"a for many discussions on the topic, and Pierre Guillon and Jarkko Kari for detailed comments on the chapter of \cite{Sa14} corresponding to the present article. The author was partially supported by the Academy of Finland Grant 131558, CONICYT Proyecto Anillo ACT 1103, by Basal project CMM, Universidad de Chile and by FONDECYT Grant 3150552.

\bibliographystyle{plain}
\bibliography{../../bib/bib}{}

\end{document}